\documentclass{article}
\usepackage[utf8]{inputenc}   
\usepackage[T1]{fontenc}      
\usepackage[english]{babel}   
\usepackage{lmodern}          

\usepackage{amsmath, amssymb, amsthm} 
\usepackage{mathtools}                
\usepackage{bm}                       

\usepackage{graphicx}
\usepackage{float}
\usepackage{caption}
\usepackage{enumitem}
\usepackage{subcaption}

\usepackage{geometry}
\usepackage{setspace} 
\usepackage{cite}    

\usepackage{hyperref}
\hypersetup{
    colorlinks = true,
    linkcolor  = blue,
    citecolor  = blue,
    urlcolor   = blue
}

\newtheorem{theorem}{Theorem}[section]
\newtheorem{lemma}[theorem]{Lemma}

\newtheorem{definition}[theorem]{Definition}

\newtheorem{remark}[theorem]{Remark}
\newtheorem*{theorem*}{Theorem}
\newtheorem*{corollary*}{Corollary}
\newtheorem*{lemma*}{Lemma}
\newtheorem*{definition*}{Definition}
\newtheorem*{remark*}{Remark}

\numberwithin{equation}{section}

\title{Sharpness and Planar Stability of the Alexandrov-Bakelman-Pucci Estimate: a Convex Geometric Approach}

\author{Antonio Porricelli}
\date{}
\begin{document}
\maketitle
\begin{abstract}
\noindent Based on a 1981 work by G. Talenti, which established a comparison principle for smooth solutions of the Monge-Ampère equation alongside related two-dimensional a priori estimates, this paper explores the connection to the Blaschke-Santalò inequality and its reverse form by K. Mahler (1939). More specifically, in this work, a quantitative version of one of Talenti's a priori estimates is proved. I exploit the interplay between Talenti's analytical framework and Mahler/Blaschke-Santalò inequalities via convex geometry, demonstrating how this link can be utilized to gain sharp insights into the Alexandrov-Bakelman-Pucci (ABP) maximum principle for elliptic PDEs, deducing a stability result in dimension two involving such estimate.
\end{abstract}
\large
\section{Introduction}

The study of sharp a priori estimates for solutions to elliptic partial differential equations has a long and celebrated history, sitting at the fertile intersection of mathematical analysis and convex geometry. A foundational milestone in this framework was established in 1981 by G. Talenti in his seminal paper \cite{talenti1981}. Investigating smooth and generalized solutions to the Dirichlet problem for the Monge-Ampère equation in the plane, Talenti proved a remarkably elegant supremum estimate. Specifically, let $\Omega \subseteq \mathbb{R}^2$ be an open bounded convex domain and let $u \in C^2(\Omega) \cap C(\overline{\Omega})$ be a concave function satisfying $u = 0$ on $\partial \Omega$. Then, the following sharp inequality holds:
\begin{equation}\label{eq:talenti_classic}
(\sup_\Omega u)^2 \leq \frac{4}{27}|\Omega|\int_\Omega \det (-D^2u) \, dx.
\end{equation}
The constant $\frac{4}{27}$ is sharp and is asymptotically attained when the domain $\Omega$ is a triangle and $u$ is an approximation of a suitable cone over $\Omega$. 

\noindent Talenti's arguments were profoundly geometric. In Section 6 ("Perimeter or area?") of his work, he pointed out that the analytical deficit in \eqref{eq:talenti_classic} is intimately tied to the dual structure of convex bodies. Indeed, while the classical PDE framework naturally considers open bounded convex domains $\Omega \subseteq \mathbb{R}^2$, their closures $\overline{\Omega}$ are convex bodies in the sense of convex geometry. Since the boundary $\partial \Omega$ has Lebesgue measure zero, the volumetric properties of $\Omega$ and $\overline{\Omega}$ are identical, rendering Talenti's setting effectively suitable to convex geometry.

\noindent The key motivation of this paper is to investigate the structural interplay between Talenti's result and Mahler's inequality. We can observe that the functional
\begin{equation}\label{eq:functional_j}
j(\Omega) = \inf \left\{ (\sup_\Omega u)^{-2} |\partial^+_u(\Omega)| \ : \; u \in C(\overline{\Omega}) \text{ is concave in } \Omega, \; u = 0 \text{ on } \partial \Omega, \text{ and } u \not\equiv 0 \right\}
\end{equation}
coincides precisely with the measure of the \textit{polar body} $\Omega^\circ_{s(\Omega)}$ of $\Omega$ with respect to its \textit{Santaló point} $s(\Omega)$. 

\noindent Through this identification, we highlight that the cornerstone inequality obtained by Talenti-which serves as the starting point to prove \eqref{eq:talenti_classic}-is equivalent to the dual Mahler and Blaschke-Santaló inequalities in the plane. Namely, if $\Omega \subseteq \mathbb{R}^2$ is an open bounded convex domain, then:
\begin{equation}
\frac{27}{4} \leq j(\Omega)|\Omega| \leq \pi^2,
\end{equation}
with equality on the left if and only if $\Omega$ is a triangle, and on the right if and only if $\Omega$ is an ellipse.

\noindent More recently, in the context of affine convex geometry, Böröczky, Makai, Meyer, and Reisner \cite{boroczky2013} proved a landmark stability version of Mahler's inequality for convex bodies in $\mathbb{R}^2$, quantifying how close a domain must be to a triangle in terms of the Banach-Mazur distance $\delta_{BM}$ whenever its volume product is close to the minimal value $\frac{27}{4}$.

\noindent In this paper, my aim is to bridge these two frameworks by establishing a novel, quantitative version of Talenti's a priori estimate \eqref{eq:talenti_classic}, thereby providing an explicit stability gauge for the Alexandrov-Bakelman-Pucci (ABP) maximum principle in terms of the geometry of the underlying domain. The main result shows that if the functional deficit in Talenti's inequality is small, then the domain $\Omega$ must be close to a triangle in the Banach-Mazur sense. Formally, we prove the following theorem:

\begin{theorem}
Let $\Omega \subseteq \mathbb{R}^2$ be an open bounded convex domain and let $u \in C^2(\Omega)\cap C(\overline{\Omega})$ be a concave function such that $u=0$ on $\partial \Omega$ and $u \not\equiv 0$. Then, for any triangle $T \subseteq \mathbb{R}^2$, the following quantitative estimate holds:
\begin{equation}\label{eq:main_stability}
\frac{\frac{4}{27}|\Omega|\displaystyle\int_\Omega\det (-D^2u)\,dx - (\sup_\Omega u)^2}{\frac{4}{27}|\Omega|\displaystyle\int_\Omega\det (-D^2u)\,dx} \ge \frac{\delta_{BM}(\Omega, T)-1}{\delta_{BM}(\Omega,T)+899}.
\end{equation}
\end{theorem}

\noindent Since the identification $j(\Omega)=|\Omega^\circ_{s(\Omega)}|$ and the stability theorem of Böröczky, Makai, Meyer and Reisner are available, the proof of \eqref{eq:main_stability} reduces to an elementary algebraic argument; the contribution of this paper lies in recognizing that this transfer is possible, not in the technical depth of any single step.
We also observe that these results can naturally be extended to functions $u$ that are not necessarily globally concave by restricting the analysis to the upper contact set of $u$, as is standard in the ABP framework (see, e.g., \cite{gilbarg2001}).
\noindent The paper is organized as follows. In Section 2, we review the necessary machinery from the theory of Monge-Ampère measures and establish the required tools from affine convex geometry, focusing on the properties of polar bodies and the Santaló point. In Section 3, we reconstruct Talenti's original setting and detail the optimization properties of cone functions under the Minkowski functional. Section 4 is devoted to the proof of the quantitative main theorem, and finally, in Section 5 we exploit its direct implications for the structural stability of the ABP maximum principle.
\section{Notations and preliminaries}
\subsection{The Monge-Ampère measure and the gradient map}
\begin{definition}
Let $\Omega \subseteq \mathbb{R}^N$ be an open set and let $u$ be a measurable function defined in $\Omega$. We define: 
\begin{equation}
    \Gamma^+_u = \{y \in \Omega : \exists \, p \in \mathbb{R}^N \text{ s.t. } u(x) \leq u(y) + p \cdot (x-y) \ \forall x \in \Omega\},
\end{equation}as the \textbf{upper contact set} of $u$. 

\end{definition}
\noindent Geometrically, $\Gamma^+_u$ represents the set of points where the graph of $u$ can be touched from above by a global supporting hyperplane. In other words, $\Gamma^+_u$ is precisely the subset of $\Omega$ where $u$ coincides with its concave envelope.
\\ \\
\noindent For a generic point $y \in \Gamma_u^+$, the supporting hyperplane may not be unique. In general, we define the \textbf{superdifferential} (or normal mapping) of $u$ at a point $y \in \Omega$ as the set of all the possible slopes of such supporting hyperplanes, namely: 
    \begin{equation}
    \partial^+_u(y)=\{p \in \mathbb{R}^N : u(x) \leq u(y) + p \cdot(x-y), \ \forall x \in \Omega\}
    \end{equation}
    and for any subset $E \subseteq \Omega$: 
    \begin{equation}
         \partial^+_u(E) = \bigcup_{y \in E}\partial^+_u(y).   
    \end{equation}

\noindent We observe that if $u$ is differentiable at $y \in \Gamma_u^+$, the superdifferential reduces to the singleton containing the classical gradient, i.e., $\partial^+_u(y)=\{\nabla u(y)\}$. Consequently, if $u$ is differentiable in the whole $\Omega$, we have $\partial^+_u(\Gamma_u^+) = \nabla u(\Gamma_u^+)$.
\\ \\
\noindent If $u \in C^{2}(\Omega)$, the gradient map $\nabla u$ is locally Lipschitz. By applying the classical area formula, we can relate the geometric volume of the image to the Hessian matrix:
    \begin{equation}\label{eq:area}
    |\partial^+_u(E)| \leq \int_{E \cap \Gamma_u^+} \det(-D^2 u(x)) \, dx
    \end{equation}
    for any measurable set $E \subseteq \Omega$. Equality holds if the gradient map is injective on $E \cap \Gamma_u^+$, which is guaranteed, for instance, if $u \in C^2(\Omega)$ and $u$ is strictly concave. More in general, if $u \in C^2(\Omega)$ is concave (not necessarily strictly), equality is still preserved by virtue of Sard's Theorem, since the set of critical values of $\nabla u$ has Lebesgue measure zero, and the flat regions where $\det D^2 u = 0$ contribute zero to both sides of the identity (see \cite{gutierrez2016}). 
    \\ \\
    \noindent  If $\Omega \subseteq \mathbb{R}^N$ is open and convex and $u \in C(\Omega)$ is a concave function, the superdifferential $\partial^+_u$ enjoys well-known continuity properties (see, e.g., \cite{gutierrez2016}).
   In particular, the image $\partial^+_u(E)$ of any Borel set $E \subseteq \Omega$ is Lebesgue measurable in $\mathbb{R}^N$. This allows us to define the \textbf{Monge-Ampère measure} associated with $u$ directly as the finite Borel measure $\mu_u : \mathcal{B}(\Omega) \longrightarrow [0, +\infty]$ given by
\begin{equation}\label{def:Monge}
\mu_u(E) = |\partial^+_u(E)|,
\end{equation}
which is well-defined and finite on compact subsets of $\Omega$.
\\ \\
The definition of the Monge–Ampère measure naturally leads to a wider class of solutions, known as generalized or Alexandrov solutions, which is closed under uniform convergence.
\begin{definition}
    Let $\mu$ be a Borel measure defined in $\Omega$, an open and convex subset 
of $\mathbb{R}^N$. The concave function $u \in C(\Omega)$ is a \textbf{generalized 
solution}, or \textbf{Alexandrov solution}, to the Monge–Ampère equation
$$\det (-D^2u) = \mu$$
if the Monge–Ampère measure $\mu_u$ associated with $u$ equals $\mu$.
\end{definition}
\subsection{Convex bodies and volume product}
In order to use convex geometry arguments to show the link between the Monge-Ampère equation and the Mahler/Blaschke-Santalò inequality, in this paragraph we recall some basic concepts that we'll need later, such as convex bodies, polar bodies of convex bodies, and the volume product.
\begin{definition}\label{def:cbody}
    A \textbf{convex body} is a non-empty convex, compact subset of $\mathbb{R}^N$ with non-empty interior.
\end{definition}

\begin{definition}
    Let $\Omega \subseteq \mathbb{R}^N$ be a convex body. We define the \textbf{polar body} of $\Omega$ with respect to a point $x_0 \in \text{int}(\Omega)$ as follows:
    \begin{equation}\label{eq:polar}
    \Omega^\circ_{x_0} = \{p \in \mathbb{R}^N \mid p \cdot (x-x_0) \leq 1, \, \forall x \in \Omega\}.
    \end{equation}
    If $0 \in \text{int}(\Omega)$, we will write $\Omega^\circ$ for the polar body with respect to the origin.
\end{definition}

\begin{remark}
    It follows by the previous definition that:
    \[
    \Omega^\circ_{x_0} = (\Omega-x_0)^\circ
    \]
\end{remark}

\noindent It can be proved that there exists a unique point $s(\Omega) \in \text{int}(\Omega)$ such that:
\[
|\Omega^\circ_{s(\Omega)}| = \min_{x \in \text{int}(\Omega)} |\Omega^\circ_x|
\]
And it is called the \textbf{Santalò point} of $\Omega$.

\begin{definition}
    The \textbf{volume product} of a convex body $\Omega \subseteq \mathbb{R}^N$ is defined as $|\Omega||\Omega^\circ_{s(\Omega)}|$.
\end{definition}

\noindent \begin{remark}\label{r:affine-invariance}
    It can easily be observed that the volume product of a convex body $\Omega \subseteq \mathbb{R}^N$ is affine invariant. More specifically, for any non-singular affine transformation $T(x) = Ax + b$ with $A \in GL(N, \mathbb{R})$ and $b \in \mathbb{R}^N$, the Santalò point satisfies the equivariance property
\[
s(T(\Omega)) = T(s(\Omega)).
\]
Furthermore, using the transformation law for the polar body $\big(T(\Omega)\big)_{s(T(\Omega))}^\circ = (A^{-1})^T \Omega_{s(\Omega)}^\circ$, one directly obtains
\[
|T(\Omega)| \left|\big(T(\Omega)\big)_{s(T(\Omega))}^\circ\right| = |\Omega| |\Omega_{s(\Omega)}^\circ|.
\]
This affine invariance allows us, without loss of generality, to assume that $0 \in \text{int}(\Omega)$ and that the origin is precisely the Santalò point of $\Omega$, i.e., $s(\Omega) = 0$.
See \cite{schneider2014} for more details.
     \end{remark} 
\noindent Now, we recall two important inequalities concerning upper and lower bounds for the volume product of convex bodies.

\begin{theorem}[Blaschke-Santalò inequality]
    Let $\Omega \subseteq\mathbb{R}^N$ be a convex body, and let $s(\Omega)$ be its Santalò point. Then it holds:
    \begin{equation}
    |\Omega||\Omega^\circ_{s(\Omega)}| \leq \omega_N^2
    \end{equation}
    where $\omega_N$ denotes the volume of the unit ball in $\mathbb{R}^N$. Furthermore, equality holds if and only if $\Omega$ is an ellipsoid.
\end{theorem}

\noindent For the lower bound, the situation is more nuanced. The long-standing Mahler's conjecture (1938) states that for every convex body $\Omega \subseteq \mathbb{R}^N$, the volume product is minimized by the simplex, yielding:
    \begin{equation}
    |\Omega||\Omega^\circ_{s(\Omega)}| \geq \frac{(N+1)^{N+1}}{(N!)^2}
    \end{equation}
If we restrict our attention to centrally symmetric convex bodies, the conjectured minimum is instead attained by the hypercube, with a lower bound of $\frac{4^N}{N!}$. 

\noindent To date, Mahler's conjecture has been fully proved in dimension $N=2$ (Mahler, 1939) and recently in dimension $N=3$ (see \cite{chen2026mahler}, for both the symmetric and non-symmetric cases), while it remains open for $N \geq 4$.
\section{Talenti's analytical framework}

In his seminal paper \cite{talenti1981}, Talenti investigated the properties of solutions to the Dirichlet problem for the Monge-Ampère equation in the plane. More precisely, let $\Omega \subseteq \mathbb{R}^2$ be a convex open set and let $f\in L^1(\Omega)$ be a given non-negative function. The fully nonlinear boundary value problem under consideration is given by:
\begin{equation}\label{eq:dirichlet_ma}
\begin{cases} 
\det D^2u = f & \text{in } \Omega, \\ 
u = 0 & \text{on } \partial \Omega, \\
u \geq 0 \text{ and concave} & \text{in } \Omega.
\end{cases}
\end{equation}
\noindent The main result of his paper is the following comparison principle for solutions of the previous problem, yielding an a priori estimate that depends on the perimeter of the set $\Omega$. 

\begin{theorem}[Talenti, 1981]
    Let $u \in C^2(\Omega) \cap C(\overline{\Omega})$ be the solution of problem \eqref{eq:dirichlet_ma} and let $v \in C^2(\Omega^\star) \cap C(\overline{\Omega^\star})$ be the solution of the symmetrized problem:
    \begin{equation}\label{eq:dirichlet_mb}
    \begin{cases} 
    \det D^2v = f^\sharp & \text{in } \Omega^\star, \\ 
    v = 0 & \text{on } \partial \Omega^\star, \\
    v \geq 0 \text{ and concave} & \text{in } \Omega^\star.
    \end{cases}
    \end{equation}
    Then, the following estimates hold:
    \begin{equation}\label{eq:comparison1}
        u^\star(x) \leq v(x), \quad \forall x \in \Omega^\star;
    \end{equation}
    \begin{equation}\label{eq:comparison2}
        \int_{\Omega} |Du|\,dx \leq \int_{\Omega^\star} |Dv|\,dx.
    \end{equation}
    Here, $\Omega^\star$ stands for the ball having the same perimeter of $\Omega$, \ $\Omega^\sharp$ stands for the ball having the same area as $\Omega$, $f^\sharp$ denotes the spherically symmetric decreasing rearrangement of $f$, extended by zero in $\Omega^\star \setminus \Omega^\sharp$ and $u^\star$ denotes the analogous perimeter-based rearrangement of $u$ (for more details about rearrangements, see \cite{talenti1981} or \cite{tso1989}). Moreover, equality holds if and only if $\Omega$ is equivalent to a ball and $f$ is radially symmetric.
\end{theorem}
\noindent Considering the fact that the solution $v$ to problem \eqref{eq:dirichlet_mb} is explicitly given by:
\begin{equation}
   v(x) = \int_{\pi|x|^2}^{\frac{P(\Omega)^2}{4\pi}} \left( \frac{1}{\pi} \int_0^s f^*(r)\,dr \right)^\frac{1}{2} \frac{ds}{\sqrt{4\pi s}},
\end{equation} 
and using the property that rearrangements preserve the supremum of a function, one obtains the following sharp a priori estimate:
\begin{equation}
    (\sup_\Omega u)^2 \leq \frac{P(\Omega)^2}{4\pi^3}\int_\Omega f(x)\,dx,
\end{equation}
with equality when $\Omega$ is a ball and $f$ is constant

\vspace{1em}
\noindent It is worth noting that both this classical result and its generalization by Tso (1989) provide estimates incorporating each quermassintegral of the domain, except for the volume. Conversely, in Section 6 of his seminal paper, Talenti derived an alternative bound involving the area of the domain, which directly connects to the Mahler inequality for the volume product of convex bodies. A key feature of Talenti's approach is the comparison of a general solution $u$ with a suitable extremal geometry constructed over the same domain. In the same paper, he obtains another a priori estimate involving the area of the set $\Omega$ via convex geometry arguments, and more specifically, a result that is equivalent to the Mahler/Blaschke-Santaló inequality in the plane. To this end, let $x_0 \in \Omega$ be a point where $u$ attains its global maximum. We introduce the cone function $k: \overline{\Omega} \to [0, +\infty)$ uniquely determined by the Minkowski functional of $\Omega$ centered at $x_0$:
\begin{equation}\label{eq:cone_definition}
k(x) = u(x_0) \left( 1 - \inf\left\{ t > 0 : x - x_0 \in t(\Omega - x_0) \right\} \right).
\end{equation}
By construction, $k$ is a concave function such that $k(x_0) = u(x_0)$ and $k = 0$ on $\partial\Omega$. Geometrically, the graph of $k$ is a ruled surface with a vertex at $(x_0, u(x_0))$. A fundamental property of this construction relies on the inclusion of their respective superdifferentials. Indeed, since $u$ is concave and vanishes on the boundary, any global supporting hyperplane for the cone $k$ at its vertex $x_0$ lies strictly below $u$ across a part of the domain $\Omega$. Consequently, by performing a vertical translation, such a hyperplane necessarily touches the graph of $u$ at some point belonging to the upper contact set $\Gamma_u^+ = \Omega$. It follows that:
\begin{equation}\label{eq:superdiff_inclusion}
\partial^+_k(x_0) \subseteq \partial^+_u(\Omega).
\end{equation}
\begin{remark}\label{rem:general-u}
The inclusion \eqref{eq:superdiff_inclusion} extends to a general (not necessarily concave) function $u\in C(\overline\Omega)$, $u=0$ on $\partial\Omega$, attaining its maximum at $x_0$, upon replacing $\Omega$ with the upper contact set $\Gamma_u^+$ throughout (see \cite{gilbarg2001}, Chapter 9).
\end{remark}
\noindent Furthermore, we observe that the cone function $k$ is differentiable at every point $x \in \Omega \setminus \{x_0\}$, where its superdifferential reduces to a single gradient vector $\partial^+_k(x) = \{\nabla k(x)\}$. Since the image of the set of differentiability points $\Omega \setminus \{x_0\}$ via the normal mapping has zero Lebesgue measure, the total Monge-Ampère measure of the cone is concentrated at the vertex. This yields the crucial geometric bound:
\begin{equation}\label{eq:measure_bound_cone}
|\partial^+_k(x_0)| = |\partial^+_k(\Omega)| \leq |\partial^+_u(\Omega)|.
\end{equation}

\noindent Motivated by this comparison mechanism, for any open bounded convex domain $\Omega \subseteq\mathbb{R}^N$ and any fixed vertex $x_0 \in \Omega$, we introduce the functional:
\begin{equation}\label{eq:functional_J}
J(x_0, \Omega) = \inf \left\{ u(x_0)^{-N} |\partial^+_u (\Omega)| \ : \; u \in C({\overline{\Omega}}) \text{ is concave in } \Omega, \; u = 0 \text{ on } \partial \Omega \text{ and }u\not \equiv 0 \text { in }\Omega \right\},
\end{equation}
and we define the global geometric invariant $j(\Omega)$ by minimizing over all possible vertex locations within the domain:
\begin{equation}\label{eq:invariant_j}
j(\Omega) = \inf_{x_0 \in \Omega} J(x_0, \Omega)
\end{equation}
$$= \inf \left\{ (\sup_\Omega u)^{-N} |\partial^+_u(\Omega)| \ : \; u \in C({\overline{\Omega}}) \text{ is concave in } \Omega, \; u = 0 \text{ on } \partial \Omega \text{ and }u\not \equiv 0 \text { in }\Omega \right\}.$$
The monotonicity and inclusion properties detailed in \eqref{eq:measure_bound_cone} guarantee that the global infimum in \eqref{eq:invariant_j} is explicitly attained when $u$ is chosen to be the optimal cone function $k$. Specifically, we have:
\begin{equation}\label{eq:brought1}
J(x_0, \Omega)=k(x_0)^{-N}|\partial^+_k(\Omega)|, \quad \forall x_0 \in \Omega; \qquad j(\Omega) = (\sup_\Omega k)^{-N}|\partial^+_k(\Omega)|,
\end{equation}
where the vertex $x_0 \in \Omega$ of the minimizing cone is precisely the point that minimizes the localized functional $J(\cdot, \Omega)$.

\section{The convex geometric perspective and the quantitative result}
Now we focus on the link with convex geometry and the Mahler inequality. To proceed, we make the following remark on the structural nature of convex bodies and open bounded convex domains.

\begin{remark}
    If $\Omega \subseteq\mathbb{R}^N$ is an open bounded convex domain, its closure $\overline{\Omega}$ is a convex body and $\Omega^\circ_{s(\Omega)}=(\overline{\Omega})^\circ_{s(\Omega)}$.
\end{remark}

\noindent We can now state and prove the main identification lemma:

\begin{lemma}
Let $\Omega \subseteq\mathbb{R}^N$ be an open bounded convex domain and $x_0 \in \Omega.$ Then:
\begin{equation}\label{eq:brought2}
J(x_0, \Omega)=|\Omega^\circ_{x_0}|.
\end{equation}
In particular, if $x_0 \in \Omega$ is the unique point that minimizes $J(\cdot, \Omega)$, then $x_0=s(\Omega)$ is the Santalò point of $\Omega$, and consequently:
\begin{equation}\label{eq:brought3}
j(\Omega)=|\Omega^\circ_{s(\Omega)}|.
\end{equation}
\end{lemma}

\begin{proof}
Let $k$ be the cone function with vertex at $x_0$. In view of \eqref{eq:brought1}, it is sufficient to prove that:
\[
\partial^+_k(x_0) = -k(x_0)\Omega^\circ_{x_0}.
\]
We establish this by showing the double inclusion:
\begin{itemize}
    \item[$\mathbf{\subseteq}$:] Let $p \in \partial^+_k(x_0)$. By definition of subdifferential for concave functions, we have:
    \[
    k(x) \leq k(x_0) + p \cdot (x - x_0), \quad \forall x \in \Omega.
    \]
    Since $k(x) \geq 0$ in $\Omega$, this implies $-p \cdot (x - x_0) \leq k(x_0) - k(x) \leq k(x_0)$, which rewrites as:
    \[
    -\frac{p}{k(x_0)} \cdot (x - x_0) \leq 1, \quad \forall x \in \Omega.
    \]
    Hence, $p \in -k(x_0)\Omega^\circ_{x_0}$.
    
    \item[$\mathbf{\supseteq}$:] Conversely, let $p \in -k(x_0)\Omega^\circ_{x_0}$. Then, by definition of the polar set:
    \[
    -\frac{p}{k(x_0)} \cdot (x - x_0) \leq 1, \quad \forall x \in \Omega.
    \]
    By density and continuity, this inequality extends to any boundary point $z \in \partial \Omega$, yielding $-p \cdot (z - x_0) \leq k(x_0)$. 
    
    Now, by the convexity of $\Omega$, for any $x \in \Omega$ there exist $\lambda \in [0,1)$ and $z \in \partial \Omega$ such that $x - x_0 = \lambda(z - x_0)$. Using the linearity of the cone function $k$ along rays from $x_0$, we have $k(x) = k(x_0)(1 - \lambda)$, or equivalently $\lambda k(x_0) = k(x_0) - k(x)$. Multiplying the boundary inequality by $\lambda \geq 0$, we get:
    \[
    -p \cdot (x - x_0) = -p \cdot \lambda(z - x_0) \leq \lambda k(x_0) = k(x_0) - k(x).
    \]
    Rearranging terms yields $k(x) \leq k(x_0) + p \cdot (x - x_0)$ for all $x \in \Omega$, which proves that $p \in \partial^+_k(x_0)$.
\end{itemize}

\noindent Now, having established that $\partial^+_k(x_0) = -k(x_0)\Omega^\circ_{x_0}$, taking the Lebesgue measure on both sides gives $|\partial^+_k(x_0)| = |k(x_0)|^N |\Omega^\circ_{x_0}|$. Combining this with \eqref{eq:brought1}, we obtain:
\[
J(x_0, \Omega) = k(x_0)^{-N} |\partial^+_k(x_0)| = |\Omega^\circ_{x_0}|.
\]
Finally, minimizing over all $x_0 \in \Omega$ yields $j(\Omega) = \inf_{x_0 \in \Omega} |\Omega^\circ_{x_0}|$. By the definition and uniqueness of the Santalò point $s(\Omega)$, the infimum is uniquely achieved at $x_0 = s(\Omega)$, which completes the proof that $j(\Omega) = |\Omega^\circ_{s(\Omega)}|$.
\end{proof}
\noindent Talenti in 1981 proved this result, which by the previous lemma, we now know that is exactly the Mahler/Blaschke-Santalò inequality in dimension two.
\begin{theorem}[Talenti, 1981]
    Let $\Omega \subseteq \mathbb{R}^2$ be an open bounded convex domain. Then it holds:
    $$\frac{27}{4}\leq j(\Omega)|\Omega|\leq \pi^2 $$
    With equality on the left when $\Omega$ is a triangle and on the right when $\Omega$ is an ellipse.
\end{theorem}
\noindent From this inequality, we can deduce that for every $u$ concave in $\Omega$ such that $u \in C({\overline{\Omega}})$, $u = 0$ on  $\partial \Omega$ and $\sup_\Omega u >0$, it holds:
$$\frac{27}{4} \leq j(\Omega) |\Omega| \leq \frac{1}{(\sup_\Omega u)^2}|\Omega||\partial^+_u(\Omega)| $$
    $$\implies (\sup_\Omega u)^2 \leq \frac{4}{27}|\Omega||\partial^+_u(\Omega)|$$
    \noindent And if $u \in C^{2}(\Omega)\cap C(\overline{\Omega})$, the area formula implies:
    $$(\sup_\Omega u)^2 \leq\frac{4}{27}|\Omega|\int_\Omega \det (-D^2u) dx$$
\begin{remark}
    Observe that the same inequalities hold also in dimension three, since Mahler's conjecture in such dimension has been proved in the recent paper \cite{chen2026mahler}. More specifically we have:
    $$(\sup_\Omega u)^3 \leq \frac{9}{64}|\Omega||\partial^+_u(\Omega)|$$
    \noindent And if $u \in C^{2}(\Omega) \cap C(\overline{\Omega})$, the area formula implies:
    \begin{equation}\label{eq:Talenti3}
    (\sup_\Omega u)^3 \leq\frac{9}{64}|\Omega|\int_\Omega \det (-D^2u) dx
    \end{equation}
\end{remark}
\noindent Now we want to prove this inequality in a quantitative form, using the concept of Banach-Mazur distance and a stability result on Mahler's inequality proved in \cite{boroczky2013}
First of all, we define the concept of Banach-Mazur distance, which measures how much two convex bodies represent an affine deformation of each other.

\begin{definition}
    Let $K, L \subseteq \mathbb{R}^N$ be two convex bodies. The \textbf{Banach-Mazur distance} between $K$ and $L$ is defined as:
    $$
    \delta_{BM}(K,L) = \inf \{ t \geq 1 \ \Big| \ \exists \, A: \mathbb{R}^N \to \mathbb{R}^N \text{ affine invertible }
    $$
    \begin{equation}
    \text{and }  x \in \mathbb{R}^N \text{ such that } L \subseteq A(K) \subseteq t L + x \}
    \end{equation}
\end{definition}
\noindent The following result quantifies how much a convex body is affine to a triangle if its volume product is close to $\frac{27}{4}$.
\begin{theorem}[B\"or\"oczky-Makai-Meyer-Reisner, 2013]
    Let $K \subseteq \mathbb{R}^2$ be a convex body and let $T$ be a triangle.
    Then, $\forall \epsilon >0$ such that $\frac{|K||K^\circ_{s(K)}|- \frac{27}{4}}{\frac{27}{4}}\leq \epsilon$ we have:
    $$\delta_{BM}(K,T) \leq 1+900 \epsilon.$$
\end{theorem}
\noindent The previous result can be stated as:
$$\delta_{BM}(K,T) \leq 1+900\big(\frac{|K||K^\circ_{s(K)}|- \frac{27}{4}}{\frac{27}{4}}\big)$$
$$\implies \frac{27}{4}\left(1 +   \frac{\delta_{BM}(K,T)-1}{900}\right)\leq |K||K^\circ_{s(K)}| $$
\begin{equation}
    \implies \frac{27}{4}\left(\frac{\delta_{BM}(K,T)+899}{900}\right)\leq |K||K^\circ_{s(K)}|
\end{equation}
\noindent And now we are ready to prove our inequality in a quantitative form.
\begin{theorem}
    Let $\Omega \subseteq \mathbb{R}^2$ be a convex body and let $u \in C(\Omega)$ a concave function such that $u=0$ on $\partial \Omega $ and $u\not\equiv0$ in $\Omega.$  Then, if $T$ is any triangle it holds:
    $$\frac{\frac{4}{27}|\Omega||\partial^+_u(\Omega)|-(\sup_\Omega u)^2}{\frac{4}{27}|\Omega||\partial^+_u(\Omega)|}\geq \frac{\delta_{BM}(\Omega, T)-1}{\delta_{BM}(\Omega,T)+899}.$$
\end{theorem}
\begin{proof}
Starting from the previous inequality and proceeding as for Talenti inequality we obtain:
$$ \frac{27}{4}\left(\frac{\delta_{BM}(\Omega,T)+899}{900}\right) \leq |\Omega||\Omega^\circ_{s(\Omega)}|=|\Omega|j(\Omega)\leq  \frac{1}{(\sup_\Omega u)^2}|\Omega||\partial^+_u(\Omega)|$$
$$ \implies \ (\sup_\Omega u)^2\leq \frac{4}{27}\left(\frac{900}{\delta_{BM}(\Omega, T)+899}\right)|\Omega||\partial^+_u(\Omega)|$$
Then we can deduce:
$$\frac{4}{27}|\Omega||\partial^+_u(\Omega)|-(\sup_\Omega u)^2\geq \frac{4}{27}|\Omega||\partial^+_u(\Omega)|-\frac{4}{27}\left(\frac{900}{\delta_{BM}(\Omega, T)+899}\right)|\Omega||\partial^+_u(\Omega)|$$
$$\implies \frac{\frac{4}{27}|\Omega||\partial^+_u(\Omega)|-(\sup_\Omega u)^2}{\frac{4}{27}|\Omega||\partial^+_u(\Omega)|}\geq 1 - \frac{900}{\delta_{BM}(\Omega, T)+899} = \frac{\delta_{BM}(\Omega, T)-1}{\delta_{BM}(\Omega,T)+899}$$
As stated.
\end{proof}
\begin{remark}
 If $u \in C^{2}(\Omega)\cap C(\overline{\Omega})$ and it satisfies the previous hypothesis, we can apply the area formula as in \eqref{eq:area} and finally obtain:
     \begin{equation}\label{eq:quant}
     \dfrac{\dfrac{4}{27}|\Omega|\displaystyle\int_\Omega\det (-D^2u)dx-(\sup_\Omega u)^2}{\dfrac{4}{27}|\Omega|\displaystyle\int_\Omega\det (-D^2u)dx}\geq \dfrac{\delta_{BM}(\Omega, T)-1}{\delta_{BM}(\Omega,T)+899}.
     \end{equation}
     
\end{remark}
\section{On the ABP maximum principle}

\subsection{The classical ABP theorem}

We conclude by analyzing the implications of this theory for the celebrated Alexandrov-Bakelman-Pucci (ABP) weak maximum principle for a specific class of second-order fully nonlinear and linear elliptic operators, providing a quantitative stability version for convex bodies in dimension two based on our previous results.

\noindent Let $\Omega \subseteq \mathbb{R}^N$ be a bounded open set and let $L$ be a linear uniformly elliptic operator of second order defined by:
\[
Lu = \sum_{i,j=1}^N a_{ij}(x) \frac{\partial^2 u}{\partial x_i \partial x_j} + \sum_{i=1}^N b_i(x) \frac{\partial u}{\partial x_i} + c(x)u.
\]
Let $A(x) = (a_{ij}(x))$ denote the coefficient matrix, and assume there exist two structural constants $0 < \lambda \le \Lambda < +\infty$ such that, for all $x \in \Omega$ and all $\xi \in \mathbb{R}^N$, the uniform ellipticity condition holds:
\[
\lambda|\xi|^2 \le \sum_{i,j=1}^N a_{ij}(x)\xi_i\xi_j \le \Lambda|\xi|^2.
\]
Setting $b = (b_i)$, the operator $L$ can be compactly expressed as $Lu = \operatorname{tr}(AD^2u) + b \cdot \nabla u + cu$. Under the standard assumptions that the normalized coefficients satisfy $\frac{a_{ij}}{\det(A)^{1/N}}, \frac{b_i}{\det(A)^{1/N}} \in L^N(\Omega)$, with $c \le 0$ almost everywhere, and the source term satisfies $\frac{f}{\det(A)^{1/N}} \in L^N(\Omega)$, the classical ABP maximum principle provides an a priori bound for solutions to the differential inequality $Lu \ge f$ in terms of the geometry of the domain and the $L^N$-norm of the data:

\begin{theorem}
Let $\Omega \subseteq \mathbb{R}^N$ be a bounded domain. Let $u \in W^{2,N}_{\text{loc}}(\Omega)$ satisfy $Lu \ge f$ in $\Omega$ (possibly in the weak sense). Then there exists a positive constant $C$, depending only on $N$, $\operatorname{diam}(\Omega)$, and $\left\|\frac{b}{\det(A)^{1/N}}\right\|_{L^N(\Omega)}$, such that:
\[
\sup_\Omega u \le \limsup_{x \to \partial\Omega} u^+ + C \left\| \frac{f^-}{\det(A)^{1/N}} \right\|_{L^N(\Gamma^+_u)}.
\]
\end{theorem}

\noindent In general, explicit formulas for the sharp constant $C$ are extremely difficult to track due to the abstract geometric covering mechanism of the normal mapping. However, in low dimensions ($N=2$ and $N=3$), Talenti's framework provides an elegant dictionary linking Monge-Ampère measures to linear operators via algebraic invariants.

\noindent From a historical perspective, the roots of this maximum principle lie in the pioneering geometric ideas of A.D. Alexandrov (1961) regarding the subdifferentials of convex functions. This geometric machinery was subsequently translated into the framework of elliptic partial differential equations by I.Ya. Bakelman (1963, 1965). Independently, C. Pucci (1966) derived uniform bounds for second-order elliptic operators by introducing the notion of extremal operators, culminating in what is now universally designated as the Alexandrov-Bakelman-Pucci (ABP) estimate. In subsequent decades, a major line of inquiry focused on sharpening the geometric dependencies embedded within the classical ABP bound. A significant breakthrough in this direction was achieved by Cabré (1995), who re-examined the analytical structure of the estimate to eliminate its rigid dependence on the linear diameter of the domain $\Omega$. Cabré's refinement successfully replaced the diameter with an intrinsic constant scaling directly with the Lebesgue measure (volume) of the domain, thereby providing a more stable and geometrically adaptive formulation of the maximum principle.

\subsection{Sharp connection between ABP and Monge-Ampère}

The structural bridge connecting the Alexandrov-Bakelman-Pucci estimate to the Monge-Ampère operator relies fundamentally on the classical arithmetic-geometric mean (AM-GM) inequality applied to matrix products.

\begin{theorem}[AM-GM inequality for matrices]
Let $A, B \in \operatorname{Sym}(N)$ be positive definite matrices. Then:
\begin{equation}\label{eq:am_gm_matrices}
\det(A)\det(B) \le \left(\frac{\operatorname{tr}(AB)}{N}\right)^N,
\end{equation}
with equality holding if and only if $A^{1/2}BA^{1/2}$ is a scalar multiple of the identity, i.e., $A^{1/2}BA^{1/2}=cI$ for some $c > 0$.
\end{theorem}

\begin{remark}
The inequality \eqref{eq:am_gm_matrices} extends to the positive semidefinite case by a standard approximation argument.
Indeed, if $A, B \in \operatorname{Sym}(N)$ are positive semidefinite, then for any $\varepsilon > 0$ the matrices $A_\varepsilon = A + \varepsilon I$ and $B_\varepsilon = B + \varepsilon I$ are strictly positive definite. Since $\det(A_\varepsilon) = \det(A) + \mathcal{O}(\varepsilon)$ and $\operatorname{tr}(A_\varepsilon) = \operatorname{tr}(A) + \varepsilon N$, we have:
\[
\begin{aligned}
\left( \det(A) + \mathcal{O}(\varepsilon) \right) \left( \det(B) + \mathcal{O}(\varepsilon) \right) &= \det(A_\varepsilon)\det(B_\varepsilon) \\
&\le \left(\frac{\operatorname{tr}(A_\varepsilon B_\varepsilon)}{N}\right)^N \\
&= \left(\frac{\operatorname{tr}(AB)+\varepsilon\operatorname{tr}(A+B)+\varepsilon^2 N}{N}\right)^N.
\end{aligned}
\]
Taking the limit as $\varepsilon \to 0^+$, we conclude that \eqref{eq:am_gm_matrices} holds for $A, B \ge 0$, with equality holding if and only if $A^{1/2}BA^{1/2} = cI$ for some $c \ge 0$.
\end{remark}

\noindent In our specific framework, we restrict our attention to the pure second-order linear elliptic operator $Lu = \operatorname{tr}(AD^2u)$. Assuming that $u$ is concave (so that $-D^2u$ is positive semidefinite), inequality \eqref{eq:am_gm_matrices} yields a sharp pointwise differential control:
\begin{equation}\label{eq:pointwise_ma_linear}
\det(-D^2u) \le \left(\frac{-Lu}{N\det(A)^{1/N}}\right)^N.
\end{equation}
Equality is achieved if and only if the diffusion matrix $A(x)$ satisfies $(A(x))^{1/2}(-D^2u(x))(A(x))^{1/2} = c(x)I$ point-by-point. This relation offers a direct path to interpret the Monge-Ampère equation as the lower envelope of linearized second-order elliptic PDEs.

\noindent By combining this pointwise matrix control with quantitative functional inequalities, we establish sharp, geometrically stable formulations of the ABP maximum principle on convex domains. In the analysis of the equality case, a technical difficulty arises since the extremal cone function fails to belong to $C^2(\Omega)$; however, this lack of regularity is seamlessly handled via a smooth approximation argument.

\noindent The first estimate in the plane was obtained by Talenti in his seminal paper \cite{talenti1981} and follows from \eqref{eq:talenti_classic} directly by applying \eqref{eq:pointwise_ma_linear}. We recall it here as our starting point:

\begin{theorem}
Let $\Omega \subseteq \mathbb{R}^2$ be an open bounded convex domain, let $A(x) \in \operatorname{Sym}^+(2)$ define a uniformly elliptic linear operator, and let $Lu = \operatorname{tr}(A D^2u)$. Let $f$ be a given datum such that $\frac{f^-}{\det(A)^{1/2}} \in C(\overline{\Omega})$, and let $u \in C^2(\Omega) \cap C(\overline{\Omega})$ satisfy $Lu \ge f$ in $\Omega$ with zero Dirichlet boundary conditions $u=0$ on $\partial \Omega$. Then the following functional estimate holds:
\begin{equation}
\sup_\Omega u \le \frac{1}{3\sqrt{3}}|\Omega|^{1/2} \left\| \frac{f^-}{\det(A)^{1/2}} \right\|_{L^2(\Gamma^+_u)}.
\end{equation}
\end{theorem}

\noindent With an identical proof, applying the recent result in \cite{chen2026mahler} and the related inequality \eqref{eq:Talenti3}, we obtain the corresponding sharp estimate in dimension three:

\begin{theorem}
If $N=3$, the following estimate holds:
\begin{equation}
\sup_\Omega u \le \frac{1}{4 \sqrt[3]{3}}|\Omega|^{1/3} \left\| \frac{f^-}{\det(A)^{1/3}} \right\|_{L^3(\Gamma^+_u)}.
\end{equation}
\end{theorem}

\noindent We point out that the previous estimates hold for uniformly elliptic operators and they are sharp as we will prove in general below, in the class of operators allowing degenerating or unbounded ellipticity, since if $\Omega$ is a triangle (or a simplex in dimension three), $u_\mu$ is an optimal smooth cone approximation, and $A(x)$ satisfies $(A(x))^{1/2}(-D^2u(x))(A(x))^{1/2} = c(x)I$ point-by-point with $c(x) \ge 0$, equality is asymptotically attained as $\mu \rightarrow 0$. A crucial observation is that standard formulations of the ABP estimate do not provide sharp bounds in terms of the volume of the domain across all dimensions. For instance, a well-known classical bound (see \cite{gilbarg2001}, Chapter 9) is expressed via the diameter of the domain:
\begin{equation}\label{eq:abp_diameter}
\sup_\Omega u \le \frac{\operatorname{diam}(\Omega)}{\omega_N^{1/N}} \left( \int_{\Gamma_u^+} \det(-D^2u) \, dx \right)^{1/N},
\end{equation}
where $\Gamma_u^+$ denotes the upper contact set of $u$. This estimate originates from constructing a comparison cone $k$ whose vertex corresponds to the maximum point $x_0 \in \Omega$ and whose base is the ball centered at $x_0$ with radius equal to $\operatorname{diam}(\Omega)$. By analyzing the subdifferential image of $k$, one obtains the inclusion:
\[
B_{\frac{u(x_0)}{\operatorname{diam}(\Omega)}}(0) \subseteq \nabla u(\Gamma_u^+).
\]
Taking the Lebesgue measure on both sides yields:
\[
\omega_N \left( \frac{\sup_\Omega u}{\operatorname{diam}(\Omega)} \right)^N \le |\nabla u(\Gamma_u^+)|.
\]
Applying the area formula to the right-hand side immediately produces \eqref{eq:abp_diameter}:
\[
\omega_N \left( \frac{\sup_\Omega u}{\operatorname{diam}(\Omega)} \right)^N \le \int_{\Gamma_u^+} \det(-D^2u) \, dx.
\] 

\noindent Although robust, estimate \eqref{eq:abp_diameter} is structurally suboptimal due to its dependence on the linear diameter, which fails to capture the intrinsic volume geometry of $\Omega$. 

\noindent A sharp alternative is achieved by constructing the comparison cone $k$ directly over the base $\Omega$ with its vertex at the maximum point $x_0$. Here, utilizing the geometric relation \eqref{eq:brought2} between the cone and the polar body $\Omega_{x_0}^\circ$, the superdifferential image satisfies:
\[
(\sup_\Omega u)^N |\Omega_{x_0}^\circ| = |\partial^+_k(x_0)| \le |\partial^+_u(\Omega)|.
\]
This finer geometric confinement leads to the following sharp theorem for the Monge-Ampère operator, establishing the ultimate connection with the polar body invariant:

\begin{theorem}\label{thm5.6}
Let $\Omega \subseteq \mathbb{R}^N$ be an open bounded convex domain and let $u \in C(\overline{\Omega})$ be a concave function such that $u=0$ on $\partial \Omega$, attaining its maximum at $x_0 \in \Omega$. Then:
\begin{equation}\label{eq:sharp_polar_abp}
\sup_\Omega u \le \frac{1}{|\Omega_{x_0}^\circ|^{1/N}} |\partial^+_u(\Omega)|^{1/N}.
\end{equation}
The estimate is sharp, in the sense that if $u$ is the geometric cone constructed over $\Omega$ with vertex at $x_0$, then equality holds. Furthermore, if $u \in C^2(\Omega) \cap C(\overline{\Omega})$, the subdifferential measure can be expressed via the Monge-Ampère determinant, yielding the sharp functional estimate:
\begin{equation}\label{eq:sharp_ma_smooth}
\sup_\Omega u \le \frac{1}{|\Omega_{x_0}^\circ|^{1/N}} \left( \int_{\Gamma^+_u} \det(-D^2u) \, dx \right)^{1/N}.
\end{equation}
In this smooth setting, equality is achieved asymptotically by a family of optimal regularized cones $u_\mu \in C^2(\Omega)$ converging uniformly to the geometric cone as the smoothing parameter $\mu \to 0$.

\noindent Furthermore, if $A(x) \in \operatorname{Sym}^+(N)$ defines a uniformly elliptic linear operator, $Lu = \operatorname{tr}(A D^2u)$, $f$ is a given datum such that $\frac{f^-}{\det(A)^{1/N}} \in C(\overline{\Omega})$, and $u \in C^2(\Omega) \cap C(\overline{\Omega})$ satisfies $Lu \ge f$ in $\Omega$ with $u=0$ on $\partial \Omega$, then the following sharp functional estimate holds:
\begin{equation}\label{eq:ABPN}
\sup_\Omega u \le \frac{1}{N|\Omega^\circ_{x_0}|^{1/N}} \left\| \frac{f^-}{\det(A)^{1/N}} \right\|_{L^N(\Gamma^+_u)}.
\end{equation}
Equality is asymptotically attained if $u$ is an optimal smooth cone approximation, and $A(x)$ satisfies $(A(x))^{1/2}(-D^2u(x))(A(x))^{1/2} = c(x)I$ point-by-point, with $c(x) \ge 0$. Moreover, the constant in the estimate is sharp in the broader class of operators allowing degenerating or unbounded ellipticity.
\end{theorem}

\begin{proof}
Fix a standard mollifier $\rho \in C_c^\infty(B_1(0))$, $\rho>0$ on $B_1(0)$, $\int_{B_1(0)} \rho(z)\,dz = 1$, and a null sequence $\delta_\mu \to 0^+$ (e.g.\ $\delta_\mu:=\mu$). For $\mu>0$ small enough that $B_\mu(x_0)\subseteq\Omega$, define
\[
f_\mu(x) := \delta_\mu + u(x_0)^N |\Omega^\circ_{x_0}| \, \mu^{-N} \rho\left(\frac{x-x_0}{\mu}\right), \qquad x \in \overline\Omega.
\]
By construction $f_\mu \in C^\infty(\overline\Omega)$ and, for each fixed $\mu$,
\[
0 < \delta_\mu \le f_\mu(x) \le \delta_\mu + u(x_0)^N|\Omega^\circ_{x_0}|\,\mu^{-N}\|\rho\|_\infty =: \Lambda_\mu < \infty \qquad \forall x \in \overline\Omega.
\]

\noindent By Alexandrov's existence theorem \cite{gutierrez2016, figalli2017}, for each $\mu > 0$ there exists a unique concave Alexandrov solution $u_\mu \in C(\overline\Omega)$ to the Dirichlet problem
\[
\begin{cases}
\det(-D^2 u_\mu) = f_\mu & \text{in } \Omega, \\
u_\mu = 0 & \text{on } \partial\Omega.
\end{cases}
\]
Since $f_\mu$ is bounded away from $0$ and $\infty$ on all of $\overline\Omega$, $\Omega$ is bounded and convex, and $u_\mu$ is constant on $\partial\Omega$, $u_\mu$ is strictly concave in $\Omega$ (equivalently, $-u_\mu$ is strictly convex; no line segment of its graph can lie in a supporting hyperplane, ruling out the degeneracy responsible for the loss of regularity in Pogorelov-type examples; see \cite{caffarelli1990localization}, stated there for the convex function $-u_\mu$). Together with $0<\delta_\mu\le f_\mu\le\Lambda_\mu$, Caffarelli's interior regularity theory \cite{caffarelli1990interior, caffarelli1990localization, caffarelli1991regularity} (applied to $-u_\mu$) then gives $u_\mu \in C^{2,\alpha}_{\mathrm{loc}}(\Omega)$, as required to apply Theorems 5.4, 5.5, and 5.6 to $u_\mu$.

\noindent The Borel measures $d\mu_\mu := f_\mu \, dx$ satisfy $\mu_\mu(\Omega) = \delta_\mu|\Omega| + u(x_0)^N|\Omega^\circ_{x_0}| \to u(x_0)^N|\Omega^\circ_{x_0}|$ and converge weakly-$*$ to the point mass $u(x_0)^N|\Omega^\circ_{x_0}|\,\delta_{x_0} = |\partial^+_k(x_0)|\,\delta_{x_0}$ as $\mu \to 0$, since $\delta_\mu|\Omega|\to0$. By the stability of Alexandrov solutions under weak-$*$ convergence of the prescribed measure with fixed boundary datum (see \cite{gutierrez2016, figalli2017}), the sequence $u_\mu$ converges uniformly on $\overline\Omega$ to the geometric cone $k$, so that $\sup_\Omega u_\mu \to \sup_\Omega k$.

\noindent Setting $A_\mu(x) := \mathrm{cof}(-D^2 u_\mu(x))$ where $-D^2u_\mu(x)$ is positive definite, and $A_\mu(x) := I$ elsewhere, we have $A_\mu(x)\in\operatorname{Sym}^+(N)$ for every $x\in\Omega$, and the pointwise matrix AM-GM equality \eqref{eq:am_gm_matrices} holds identically throughout $\Omega$ by construction. This establishes the sharpness of the constants in Theorems 5.4, 5.5, and 5.6, within the broader class of operators permitting degenerating or unbounded ellipticity, on every compact subset of $\Omega$.

\noindent Finally, since $\det A_\mu = \det(-D^2u_\mu)^{N-1} = f_\mu^{N-1} \sim \mu^{-N(N-1)} \to +\infty$ near $x_0$ as $\mu\to0$, at least one eigenvalue of $A_\mu$ becomes unbounded. Whether these sharp constants can be attained or approached within the restricted class of uniformly elliptic operators with fixed bounds $0<\lambda I\le A(x)\le\Lambda I$ on all of $\Omega$ remains an open problem.
\end{proof}

\begin{remark}
Regarding the sharpness of the ABP estimate involving the volume of the domain in general dimensions, the problem remains open and intrinsically linked to the sharp lower bound of the Mahler volume product for $N \ge 4$. In general, there exists a constant $C_N > 0$ such that, if $\Omega \subseteq \mathbb{R}^N$ is a convex body, then:
\[
|\Omega||\Omega^\circ_{s(\Omega)}| \ge C_N.
\]
Since $s(\Omega)$ minimizes $y \mapsto |\Omega^\circ_y|$, we have $|\Omega^\circ_{x_0}| \ge |\Omega^\circ_{s(\Omega)}| \ge C_N / |\Omega|$ for the actual maximum point $x_0$, regardless of its location. Applying this to \eqref{eq:ABPN}, we immediately obtain:
\begin{equation}\label{eq:ABPN_volume}
\sup_\Omega u \le \frac{1}{N C_N^{1/N}}|\Omega|^{1/N} \left\| \frac{f^-}{\det(A)^{1/N}} \right\|_{L^N(\Gamma^+_u)}.
\end{equation}

\noindent Furthermore, an interesting problem could be to characterize all cases of equality to obtain a rigidity result on such inequalities in each dimension.
\end{remark}

\noindent To conclude, regarding quantitative geometric stability, our novel result in dimension two follows directly from \eqref{eq:quant} by applying \eqref{eq:am_gm_matrices}; the reduction to the upper contact set in the general (not necessarily concave) case is precisely the content of Remark \ref{rem:general-u}, recalling that $\partial^+_u(y)=\emptyset$ for $y\notin\Gamma_u^+$ by definition, so that $|\partial^+_u(\Omega)|=|\partial^+_u(\Gamma^+_u)|$.
\begin{theorem}\label{thm:final-stability}
Under the hypotheses of Theorem \ref{thm5.6} on $\Omega, f$ and $A$, let $u\in C^2(\Omega)\cap C(\overline\Omega)$ be not necessarily concave (see Remark \ref{rem:general-u}), $N=2$ and $T$ be any triangle in the plane, then the functional deficit in the sharp ABP maximum principle controls the geometric distance of the domain from the optimal triangular shape via the Banach-Mazur distance $\delta_{BM}$ as follows:
\begin{equation}
\frac{\frac{1}{27}|\Omega|\left\|\frac{f^-}{\det(A)^{1/2}}\right\|^2_{L^2(\Gamma^+_u)} - (\sup_\Omega u)^2}{\frac{1}{27}|\Omega|\left\|\frac{f^-}{\det(A)^{1/2}}\right\|^2_{L^2(\Gamma^+_u)}} \ge \frac{\delta_{BM}(\Omega, T)-1}{\delta_{BM}(\Omega,T)+899}.
\end{equation}
\end{theorem}

\paragraph{Funding}
The author was partially supported by Gruppo Nazionale per l’Analisi Matematica, la Probabilità e le loro Applicazioni
(GNAMPA) of Istituto Nazionale di Alta Matematica (INdAM).

\paragraph{Competing Interests}
I declare that I have no financial and personal relationships with other people or organizations.
\newpage

\nocite{*}
\bibliographystyle{plain} 
\bibliography{references}
\bigskip
\textit{E-mail address: } 
\texttt{antonio.porricelli2@unina.it}
\newline
\textsc{Dipartimento di Matematica e Applicazioni "R. Caccioppoli", Università degli Studi di Napoli Federico II, Via Cintia, Complesso Universitario Monte S. Angelo, 80126, Napoli, Italy.}
\end{document}